\documentclass[psamsfonts]{amsproc}
\usepackage{mathtools}
\usepackage{dsfont}
\usepackage{upref}
\usepackage{amssymb}
\usepackage[all]{xy}
\newtheorem{thm}{Theorem}[section]

\newtheorem{lem}[thm]{Lemma}
\newtheorem{prop}[thm]{Proposition}
\theoremstyle{definition}
\newtheorem{defn}[thm]{Definition}
\theoremstyle{remark}

\title[A Novel Supergeometric Generalization of Grassmannians]{A Novel Supergeometric Generalization of Grassmannians}

\author[Bahadorykahlily]{F. Bahadorykhalily}
\address{Department of Mathematics, Institute for Advanced Studies in Basic Sciences (IASBS), No. 444, Prof. Yousef Sobouti Blvd.
P. O. Box 45195-1159 Zanjan Iran, Postal Code 45137-66731}
\curraddr{}
\email{f.bahadory@iasbs.ac.ir}
\thanks{}

\author[Mohammadi]{M. Mohammadi}
\address{Department of Mathematics, Institute for Advanced Studies in Basic Sciences (IASBS), No. 444, Prof. Yousef Sobouti Blvd.
P. O. Box 45195-1159 Zanjan Iran, Postal Code 45137-66731}
\curraddr{}
\email{moh.mohamady@iasbs.ac.ir}
\thanks{}

\author[Varsaie]{S. Varsaie}
\address{Department of Mathematics, Institute for Advanced Studies in Basic Sciences (IASBS), No. 444, Prof. Yousef Sobouti Blvd.
P. O. Box 45195-1159 Zanjan Iran, Postal Code 45137-66731}
\curraddr{}
\email{varsaie@iasbs.ac.ir}
\thanks{}

\date{}

\begin{document}

\begin{abstract}
A new generalization of Grassmannians to supergeometry, different from the well known supergrassmannian, is introduced. These are constructed by gluing a finite number of copies of a $ \nu $- domain, i.e. a superdomain with an odd involution, say $ \nu $, on their structure sheaf considered as a sheaf of $ C^\infty_{\mathbb{R}^m} $-modules.\\
\smallskip
\noindent \textbf{Keywords.} supermanifolds,  supergrassmannians, super vector bundles.
\end{abstract}
\subjclass[2010]{58A50}

\maketitle

\section*{Introduction}\label{int}
Many different constructions in common geometry have found proper analogues
in supergeometry such as tangent vectors, tangent bundles, vector
bundles, differential forms, etc. But Chern classes has not found a satisfactory
generalization in supergeometry. The Chern classes are cohomology elements
which may be associated to isomorphism classes of vector bundles. In physics,
these classes are related to special sort of quantum numbers called topological
charges \cite{last}. These classes may be defined by homotopy classification of vector bundles approach.

From this approach, one may see that for generalizing the Chern
classes, it is necessary to generalize Grassmannians as well as homotopy classification of vector bundles theorem. There exists a well known generalization of Grassmanian to supergeometry,
called supergrassmanian. Supergrassmannian, introduced in \cite{f41} and Grassmannian, in some sense, are homotopy equivalent. Therefore, cohomology group associated to
supergrassmannian is equal to that of Grassmannian. In other words, the former group contains no
information about superstructure. For more information see \cite{Afshari}. But there is a different generalization  called $ \nu $-Grassmanian, c.f. section ~\eqref{sec2},  which plays a main role for homotopy classification of super vector bundles. In addition, through $ \nu $-Grassmanians, one may associate an element in $ \mathbb{Z}_2 $- graded cohomology group to each isomorphism class of super vector bundles. These elements may be considered as analogues of Chern classes in super geometry. To see these results in special cases for $ \nu $-projective spaces refer to \cite{Afshari}, \cite{Roshandel}, \cite{v3} and \cite{v5}.

In this paper, we deal with introduction and construction of $ \nu $-Grassmanians. Other relevant results such as homotopy classification of super vector bundles or $ \mathbb{Z}_2 $- graded cohomology group associated to isomorphism class of super vector bundles will be discussed in forthcoming papers.

In this paper, the first  section contains a brief summary of supermanifolds and also $\nu$-domains. Proposition \eqref{prop1.1} shows that an odd involution always exists.

In section \ref{sec2}, we introduce $\nu$-Grassmannians as a generalization of the Grassmannians. For this, we construct a $\nu $-Grassmannian by gluing $ \nu $-domains. The existence of a canonical super vector bundle over a $ \nu $-Grassmannian is showed in Theorem \eqref{thm}.

\section{Preliminaries}\label{pre}
A supermanifold of dimension  $ m|n $ is a superspace, namely $ (X,\mathcal{O}) $, which is locally isomorphic to a superdomain $ (\mathbb{R}^m, C^\infty_{\mathbb{R}^m} \otimes \wedge \mathbb{R}^n) $, where by $ C^\infty_{\mathbb{R}^m} $ we mean the sheaf of smooth functions on $ \mathbb{R}^m $.
Let $ \mathcal{O}= \mathcal{O}_0 \oplus \mathcal{O}_1 $ and let $ a $ be an element of $ \mathcal{O}_{\delta} $; then $a$ is called a homogeneous element of degree $ p(a)=\delta $.
A morphism between two supermanifolds, $ (X, \mathcal{O}_X) $ and $ (Y, \mathcal{O}_Y) $, is a pair $ \psi =(\tilde{\psi}, \psi^*) $ such that $ \tilde{\psi}: X \rightarrow Y $ is a continuous map and $ \psi^* :\mathcal{O}_Y \rightarrow \psi_*(\mathcal{O}_X) $ is a morphism between sheaves of super commutative local rings.

Let $ \mathcal{J}_X $ be the sheaf of nilpotents in $ \mathcal{O}_X $. Obviously, $ \mathcal{J}_X $ is a sheaf of ideals in $ \mathcal{O}_X $. Thus $ \tilde{\mathcal{O}}_X:=\dfrac{\mathcal{O}_X}{\mathcal{J}_X} $ is a sheaf of rings and is isomorphic to $ C^\infty_X $. Thus $ (X,\tilde{\mathcal{O}}_X) $ is a common smooth manifold and is called \textbf{reduced manifold}  associated to $ (X,\mathcal{O}_X) $.

By a $ \nu $- domain, we mean a superdomain $ (\mathbb{R}^m, C^\infty_{\mathbb{R}^m} \otimes \wedge \mathbb{R}^n) $ with an odd involution, say $ \nu $, on the structure sheaf considered as a sheaf of $ C^\infty_{\mathbb{R}^m} $-modules, i.e.
\begin{equation*}
\nu^2=1.
\end{equation*}
A superdomain $ \mathbb{R}^{m|n}=(\mathbb{R}^m, C^\infty_{\mathbb{R}^m}\otimes \wedge \mathbb{R}^n) $ is a $ \nu $-domain if and only if its structure sheaf carries a  $C^{\infty}_{\nu_0}$-module structure where $ C^{\infty}_{\nu_0}=C^{\infty}_{\mathbb R}[\nu_0]$ is a ring generated by $\nu_0$ such that
\begin{equation*}
\nu_0^2=1, \; \nu_0(\wedge^o \mathbb{R}^n)\subseteq C^{\infty}(\mathbb{R}^m)|_0\otimes\wedge^e \mathbb{R} , \; \nu_0(\wedge^e \mathbb{R}^n)\subseteq C^{\infty}(\mathbb{R}^m)|_0\otimes\wedge^o \mathbb{R}
\end{equation*} 
where $ \wedge ^{\circ}\mathbb{R}^n= \sum \wedge^{2t+1}\mathbb{R}^n $ and $ \wedge^e \mathbb{R}^n=\sum \wedge^{2t} \mathbb{R}^n $ are odd and even parts of the $\mathbb{Z}_2$-graded ring  $ \wedge \mathbb{R}^n $ respectively.
Indeed, if $\mathbb{R}^{m|n}$ is a $\nu$-domain then for each element $p$ in $\wedge{\mathbb R}^n$, set $\nu_0 p:= \nu(p)$. 

Let $\mathbb{R}_{\nu}^{m|n}$ denote a $\nu$- domain. In the next proposition, the existence of such an involution is shown. For more details see \cite{v4}.
\begin{prop}\label{prop1.1}
	for $ n\geq 1 $, every superdomain $ \mathbb{R}^{m|n} $ is a $ \nu $- domain.
\end{prop}
\begin{proof}
	Obviously, $ \wedge \mathbb{R}^n $, as a real super vector space, has an ordered basis, say $ \mathcal{B} $, which its  first $ 2^{n-1} $ elements are even and the others are odd. One may easily show the existence of an odd linear transformation $ T:\wedge \mathbb{R}^n \rightarrow \wedge \mathbb{R}^n $ with the following properties:
	\begin{equation*}
	T(\wedge^o \mathbb{R}^n)\subset \wedge^e \mathbb{R}^n, \ T(\wedge^e \mathbb{R}^n) \subset \wedge^o \mathbb{R}^n, \ T^2=id.
	\end{equation*}
	Indeed, let $ A $ be an invertible matrix of rank $ 2^{n-1} $ and let $ T $ be the unique linear transformation with matrix representation in the ordered basis $ \mathcal{B} $, as follows:
	\begin{equation*}
	\left[
	\begin{array}{c|c}
	0 & A \\
	\hline
	A^{-1} & 0
	\end{array}
	\right].
	\end{equation*}
	Now, consider the odd linear transformation $ \nu:C^\infty_{\mathbb{R}^m}\otimes \wedge \mathbb{R}^n \rightarrow C^\infty_{\mathbb{R}^m}\otimes \wedge \mathbb{R}^n $ defined by
	\begin{equation*}
	a \otimes b \longmapsto a \otimes T(b); \ a\in C^\infty_{\mathbb{R}^m}, b \in \wedge \mathbb{R}^n .
	\end{equation*}
	Obviously, $ \nu $ is an odd involution on the sheaf $ C^\infty_{\mathbb{R}^m}\otimes \wedge \mathbb{R}^n $.
\end{proof}

\section{Construction of $\nu$-Grassmannian}\label{sec2}

By a $\nu$-Grassmannian, denoted by $_\nu G_{k|l}(m|n) $,  we mean a supermanifold which is constructed by gluing $\nu$-domains
$ \big(\mathbb{R}^{\alpha}, C^{\infty}_{\mathbb{R}^{\alpha}} \otimes \wedge \mathbb{R}^{\beta} \big) $ 
where $ \alpha=k(m- k)+ l(n-l) $ and $ \beta=l(m- k)+k(n- l) $.\\
Let $I\subset \{1, \cdots, m \}$ and $ R \subset \{1, \cdots, n\}$ be sorted subsets in ascending order, with $ p $ and $ q $ elements respectively such that $p+q=k+l$. The elements of $ I $ are called even indices and the elements of $ R $ are called odd indices. In this case, $I|R$ is called a $p|q$-index. Set
$U_{I|R}=\mathbb{R}^{\alpha}, \mathcal{O}_{I|R}=\,  C^{\infty}_{\mathbb{R}^{\alpha}}\otimes \wedge \mathbb{R}^{\beta}.$
If $ p=k $ then $ I|R $ is called a \textbf{standard index} and $ (U_{I|R},O_{I|R}) $, or $ U_{I|R} $ for brevity, is called a \textbf{Standard domain} and otherwise they are called \textbf{non standard index} and \textbf{non standard domain} respectively. Decompose any even super matrix into four blocks, say $ B_1, B_{2}, B_3, B_4 $. Upper left and lower right blocks, $ B_1, B_4 $ are $ k\times m $ and $ l\times n $ matrices respectively. They are called even blocks. Upper right and lower left blocks, $ B_2, B_3 $ are $ k\times n $ and $ l\times m $ matrices. They are called odd blocks. In addition, by even part we mean the blocks $ B_1, B_3 $ and by odd parts we mean the blocks $ B_2, B_4 $. Blocks, $ B_1, B_4 $ are filled with even elements and blocks, $ B_2, B_3 $ are filled with odd elements. By \textbf{divider line}, we mean the line which separates odd and even parts.

Let each domain $ U_{I|R} $ be labeled by an even $ k|l\times m|n $ supermatrix, say $ A_{I|R} $ with four blocks, $ B_1, B_{2}, B_3, B_4 $ as above.  Except for columns with indices in $ I \cup R $, which together form a minor denoted by $ M_{I|R}(A_{I|R}) $, the even and odd blocks are filled from up to down and left to right by $ x_a^I, e_b^I $, the even and odd coordinates of $ U_{I|R} $ respectively, i.e. $ (x^I_a)_{1\leq a\leq \alpha} $ is a coordinate system on $ \mathbb{R}^\alpha $ and $ \{e^I_b\}_{1\leq b\leq \beta} $ is a basis for $\mathbb{R}^{\beta} $.
 This process imposes an ordering on the set of coordinates. If $ p=k $ then $ M_{I|R}A_{I|R} $ is supposed to be an identity matrix.

Let $ I|R $ and $ J|S $ be two standard indices and let $  M_{J|S}(A_{I|R}) $ be the minor consisting of columns of $ A_{I|R} $ with indices in $ J\cup S $. By $ U_{{I|R},{J|S}} $ we mean the set of all points of $ U_{I|R} $, on which $ M_{J|S}(A_{I|R}) $ is invertible. Obviously $ U_{{I|R},{J|S}} $ is an open set.  

For example, let $ I=\{2\}, R=\{1,3\} $ and let $ I|R $ be a $ 1|2 $-index in $ _\nu G_{1|2}(3|3) $. In this case, the set of coordinates of $\mathcal{O}_{I|R} $is 
\[
\lbrace x_1, x_2, x_3, x_4; e_{1}, e_{2}, e_{3}, e_{4}, e_{5}\rbrace
\],

and $ A_{I|R} $ is:
\begin{equation*}
\left[
\begin{array}{ccc|ccc}
x_1 & 1 & x_2 & 0 & e_5 & 0 \\
\hline
e_1 & 0 & e_3 & 1 & x_3 & 0 \\
e_2 & 0 & e_4 & 0 & x_4 & 1 \\
\end{array}
\right].
\end{equation*}
Thus $ \{x_1, e_1, e_2, x_2, e_3, e_4, e_5, x_3, x_4\} $ is the corresponding total ordered set of generators.

If $ p\neq k $, then $ M_{I|R}(A_{I|R}) $ is a $ k|l\times p|q $  supermatrix as follows:\\
Let $ M_{I|R}(A_{I|R}) $ be partitioned into four blocks $ B_i, i=1,2, 3, 4 $, as above. All entries of this supermatrix except diagonal entries are zero. In addition, diagonal entries are equal to 1 if they place in $ B_1 $ and $ B_4 $ and are equal to $ 1\nu $ if they place in $ B_2 $ and $ B_3 $, where $ 1\nu $ is a formal  symbol. One may consider it as 1 among odd elements. Nevertheless we learn how to deal with it as we go further. Such supermatrix is called a \textbf{non-standard identity}. All places in $ A_{I|R} $ except for $ M_{I|R}(A_{I|R}) $ are filled by coordinates $ x^I $ and $ e^I $ according to the ordering, as stated above, from up to down and left to right.  In this process, if an even element, say $ x $, places in odd part then it is replaced by  $ \nu(x) $ and if an odd element, say $ e $, places in even part then it is replaced by  $ \nu(e) $.

As an example, consider $ _\nu G_{1|2}(3|3) $ and let $ I=\emptyset, R=\{1,2,3\} $, so $ I|R $ is an $ 0|3 $-index. In this case, $ A_{I|R} $ is as follows:
\begin{equation*}
\left[
\begin{array}{ccc|ccc}
x_1 & x_2 & \nu(e_5) & 1\nu & 0 & 0 \\
\hline
e_1 &  e_3 & \nu(x_3) & 0 & 1 & 0 \\
e_2 &  e_4 & \nu(x_4) & 0 & 0 & 1 \\
\end{array}
\right].
\end{equation*}
Also if $ J=\{2,3\}, S=\{1\} $, then $ J|S $ is a $2|1$-index and $A_{J|S}$ is as follows:
\begin{equation*}
\left[
\begin{array}{ccc|ccc}
x_1 & 1 & 0 & 0 & \nu(x_2) & e_5 \\
\hline
e_1 &  0 & 1\nu & 0 & \nu(e_3) & x_3 \\
e_2 &  0 & 0 & 1 & \nu(e_4) & x_4 \\
\end{array}\right].
\end{equation*}
Now, let us start by constructing morphisms $g_{I|R,J|S}$ through which the $ \nu $- domains $ U_{I|R} $ may be glued together. Set
\begin{equation}\label{equg}
g_{I|R,J|S} =(\tilde{g}_{I|R,J|S}, g^*_{I|R,J|S}): (U_{I|R,J|S}, \mathcal{O}_{I|R}|_{U_{I|R,J|S}}) \rightarrow (U_{J|S,I|R},\mathcal{O}_{J|S}|_{U_{J|S,I|R}}). 
\end{equation}
First of all we introduce $ g^*_{{I|R}.{J|S}} $. For this, we should consider  two cases:
\begin{enumerate}
	\item[Case 1:]
	If both domains are standard domains, then the transition map 
	\begin{equation*}
	g^*_{I|R,J|S}:  \mathcal{O}_{J|S}|_{U_{J|S,I|R}}\rightarrow  \mathcal{O}_{I|R}|_{U_{I|R,J|S}},
	\end{equation*}
	is obtained from the pasting equation:
	\begin{equation*}
	D_{J|S}\bigg(\big(M_{J|S}(A_{I|R})\big)^{-1}A_{I|R}\bigg)=D_{J|S}A_{J|S}.
	\end{equation*}
	Where $ D_{J|S}A_{J|S} $ is a matrix which remains after omitting $ M_{J|S}(A_{J|S}) $. 
	This equation defines $ g^*_{{I|R} ,{J|S}} $, for each entry of $ D_{J|S}(A_{J|S}) $, to be a rational expression in generators of $ \mathcal{O}_{I|R} $. This determines $ g^*_{{I|R},{J|S}} $ as a unique morphism (\cite{f31}, Theo. 4.3.1).
	Clearly, this map is defined whenever $ M_{J|S}(A_{I|R}) $ is invertible.
	\item[Case 2:]
	Let $ U_{I|R} $ be an arbitrary domain, and $ U_{J|S} $ be a non-standard domain, and let $ A_{I|R} $ and $ A_{J|S} $ be their labels respectively. Moreover, let $ J|S $ be a $ p|q $-index such that $ p\neq k $. In this case, $M_{J|S}(A_{I|R})$ is a $ k|l\times p|q $ matrix. Let $ M'_{J|S}(A_{I|R}) $ be a $ k|l\times k|l $ supermatrix associated to $ M_{J|S}(A_{I|R}) $ as follows:\\
	Consider the columns of non-standard identity $ M_{J|S}A_{J|S} $ which contains $ 1\nu $.  Move the columns in $ M_{J|S}(A_{I|R}) $ with the same indices as the columns in $  M_{J|S}(A_{J|S}) $ which contain $ 1\nu $ to another side of the divider line of $  M_{J|S}(A_{I|R}) $ and replace each entry, say $ a $, in these columns with $ \nu(a) $. The resulting matrix is denoted by $ M'_{J|S}(A_{I|R}) $.
\end{enumerate}
For example in $_\nu G_{1|2}(3|3)$ suppose $I=\{2\}, R=\{1,3\}, J=\{2,3\}, S=\{1\}$ , so $ I|R $ is a $1|2$-index and $J|S$ is a $2|1$-index. We have
\begin{equation*}
A_{I|R}=\left[
\begin{array}{ccc|ccc}
x_1 & 1 & x_2 & 0 & e_5 & 0 \\
\hline
e_1 & 0 & e_3 & 1 & x_3 & 0 \\
e_2 & 0 & e_4 & 0 & x_4 & 1 \\
\end{array}
\right], \quad 
A_{J|S}=\left[
\begin{array}{ccc|ccc}
x_1 & 1 & 0 & 0 & \nu(x_2) & e_5 \\
\hline
e_1 &  0 & 1\nu & 0 & \nu(e_3) & x_3 \\
e_2 &  0 & 0 & 1 & \nu(e_4) & x_4 \\
\end{array}\right],
\end{equation*}
\begin{equation*}
M_{J|S}A_{I|R}=\left[
\begin{array}{cc|c}
1 & x_2 & 0 \\
\hline
0 & e_3 & 1 \\
0 & e_4 & 0 \\
\end{array}
\right], \qquad 
M'_{J|S}A_{I|R}=\left[
\begin{array}{c|cc}
1 & \nu x_2 & 0 \\
\hline
0 & \nu e_3 & 1 \\
0 & \nu e_4 & 0 \\
\end{array}
\right].
\end{equation*}
Let $ I|R $ be a standard index and $ J|S $ be a non standard index and let $M'_{J|S}(A_{I|R})$ be as above. By $ U_{{I|R},{J|S}} $ we mean the set of all points of $ U_{I|R} $, on which
$ M'_{J|S}(A_{I|R}) $ is invertible. Obviously, $ U_{{I|R},{J|S}} $ is an open set.  

Now, we can define a coordinate transformation:
\begin{equation*}
g^*_{I|R,J|S}:  \mathcal{O}_{J|S}|_{U_{J|S,I|R}}\rightarrow  \mathcal{O}_{I|R}|_{U_{I|R,J|S}}.
\end{equation*}
This map is obtained from the following equation:
\begin{equation*}
D_{J|S}\bigg(\big(M'_{J|S}(A_{I|R})\big)^{-1}A_{I|R}\bigg)=D_{J|S}A_{J|S}.
\end{equation*}
It can be shown that the sheaves on $ U_{I|R}$ and $U_{J|S} $ can be glued through these maps. By (\cite{f31}, page 135), we have to show the next proposition.
\begin{prop}\label{prop1}
	Let $ g^*_{I|R, J|S} $ be as above, then
	\begin{align*}
	&1.\, g^*_{I|R,I|R}=id.\\
	&2.\, g^*_{I|R,J|S}\circ g^*_{J|S,I|R}=id.\\
	&3.\, g^*_{I|R,J|S}\circ g^*_{J|S,T|P}\circ g^*_{T|P,I|R}=id.
	\end{align*}
\end{prop}
\begin{proof}
	\begin{enumerate}
		\item[1.] For the first equation, note that the map $ g^*_{I|R,I|R} $ is obtained from the equation
		\begin{equation*}
		D_{I|R}\bigg((M'_{I|R}A_{I|R})^{-1}A_{I|R}\bigg)=D_{I|R}A_{I|R}.
		\end{equation*}
		In this equality, for each $ I|R $, the matrix $ M'_{I|R}A_{I|R} $ is identity. So we have the following equality:
		\begin{equation*}
		D_{I|R}A_{I|R}=D_{I|R}A_{I|R}.
		\end{equation*}
	This shows that $ g^*_{{I|R},{I|R}}= id $.
		\item[2.]
		For the second equality, let $ J|S $ be a $ p|q $-index such that $ p\neq k $. Assume  $ p>k $, so $ g^*_{I|R,J|S} $ is defined by the pasting equality:
		\begin{equation*}
		D_{J|S}\bigg((M'_{J|S}A_{I|R})^{-1}A_{I|R}\bigg)=D_{J|S}A_{J|S}.
		\end{equation*}
		For brevity, we use $ Z $ instead of  $ M'_{J|S}A_{I|R} $, we get
		\begin{equation*}
		D_{J|S}(Z^{-1}A_{I|R})=D_{J|S}A_{J|S}.
		\end{equation*}
		By $ [A_{I|R}]_{J|S} $ we mean a matrix which is obtained from $ A_{I|R} $ after replacing $ M_{J|S}A_{I|R} $ with a non- standard identity.\\
		One may see that $  g^*_{I|R,J|S}\circ g^*_{J|S,I|R} $  is obtained by the following equality:
		\begin{equation}\label{eq1}
		\bigg({M^\prime}_{I|R}\bigg[Z^{-1}A_{I|R}\bigg]_{J|S}\bigg)^{-1}D_{I|R}\bigg[Z^{-1}A_{I|R}\bigg]_{J|S}=D_{I|R}A_{I|R}.
		\end{equation}
		If $ C_t, 1\leq t\leq m+n $, denotes the t-th column of $ A_{I|R} $, then
		\begin{equation}
		Z:= M^{\prime}_{J|S}A_{I|R}=[C_{j_1},C_{j_2}, \cdots, C_{j_k}|\nu C_{j_{k+1}},\cdots, \nu C_{j_p},C_{s_1},
		C_{s_2},\cdots, {C}_{s_q}],
		\end{equation}
		where $I=\{j_1, ..., j_p\}$ and $R=\{s_1, ..., s_q\}$.
		Therefore the supermatrix $ \big[Z^{-1}A_{I|R}\big]_{J|S} $ is as follows:
		\begin{multline*}
		\big[Z^{-1}C_1,\cdots, Id_{j_1},\cdots, Id_{j_k}, Id_{j_{k+1}},\cdots,Id_{j_p}, \cdots, Z^{-1}C_m \big\vert\\ 
		Z^{-1}C_{m+1},\cdots, Id_{s_1}, \cdots, Id_{s_q},\cdots, Z^{-1}C_{m+n}\big].
		\end{multline*}
		where $ Id_{j_r} $ is a columnar supermatrix with the only one nonzero entry equals $ 1 $ on $ r $-th row if $ 1\leq r\leq k $,  and equals $ 1\nu $ if $ k+1\leq r\leq p $. In addition, $Id_{s_t}$ is a columnar supermatrix with the only one nonzero entry equals to $ 1 $ on the $ (k+t)- th $ row for each $ 1\leq t\leq q $. So one has
		\begin{multline*}
		\big[Z^{-1}A_{I|R}\big]_{J|S}= Z^{-1}[C_1, \cdots, Z(Id_{j_1}),  \cdots, Z(Id_{j_k}), \cdots, Z(Id_{j_p}), \cdots, C_m \big\vert\\
		C_{m+1}, \cdots, Z(Id_{s_1}), \cdots, Z(Id_{s_q}),\cdots, C_{m+n}].
		\end{multline*}
		Due to the definition of $Id_{j_r}$ and $Id_{s_t}$ and also the following rule:
		\begin{equation*}
		z.{1\nu} = \nu(z),
		\end{equation*}
		where $z$ is an arbitrary element of $\mathcal O$. The following equalities hold:
		\begin{flalign*}
		&Z(Id_{j_\alpha})=C_{j_\alpha} ,\quad \forall 1\leq \alpha \leq k,\\
		&Z(Id_{j_{r}})=(\nu C_{j_{r}})1\nu=C_{j_{r}},\quad \forall k+1\leq r \leq p,\\
		&Z(Id_{s_t})=C_{s_t},\quad \forall 1\leq t \leq q.
		\end{flalign*}
		Therefore, we have:
		\begin{equation*}
		[Z^{-1}A_{I|R}]_{J|S}=Z^{-1}A_{I|R}.
		\end{equation*}
		So for the left side of equation \eqref{eq1}, one has
		\begin{align*}
		\bigg(M'_{I|R}\big[Z^{-1}A_{I|R}\big]_{J|S}\bigg)^{-1} & D_{I|R}\big[Z^{-1}A_{I|R}\big]_{J|S}\\ & =\bigg(M'_{I|R}(Z^{-1}
		A_{I|R})\bigg)^{-1}D_{I|R}(Z^{-1}A_{I|R})\\
		& = \bigg(Z^{-1}M'_{I|R}A_{I|R}\bigg)^{-1}Z^{-1}D_{I|R}A_{I|R}\\ & =\bigg(M'_{I|R}A_{I|R}\bigg)^{-1}Z
		Z^{-1}D_{I|R}A_{I|R}=D_{I|R}A_{I|R}.
		\end{align*}
	It follows that the map $ g^*_{I|R,J|S}\circ g^*_{J|S,I|R} $ is identity.
		\item[3.]
		To prove the third equality, it is sufficient to show that the map $ g^*_{I|R,J|S}og^*_{J|S,T|P} $ is defined by the following equation:
		\begin{equation*}
		D_{T|P}((M'_{T|P}A_{I|R})^{-1}A_{I|R})=D_{T|P}A_{T|P} .
		\end{equation*}
		Note that the map $ g^*_{I|R,J|S} $ is obtained by
		\begin{equation}\label{eq2}
		(M'_{J|S}A_{I|R})^{-1}D_{J|S}A_{I|R}=D_{J|S}A_{J|S}.
		\end{equation}
So replacing $ A_{J|S} $ by the left hand side of \eqref{eq2} in pasting equality defining $ g^*_{{J|S} , {T|P}} $ and by setting $ Z:= M'_{J|S}A_{I|R} $, one gets
		\begin{equation*}
		D_{T|P}(M'_{T|P}[Z^{-1}A_{I|R}]_{J|S})^{-1}[Z^{-1}A_{I|R}]_{J|S})=D_{T|P}A_{T|P}.
		\end{equation*}
		An analysis similar to that in case 2 shows that  $ [Z^{-1}A_{I|R}]_{J|S}=Z^{-1}A_{I|R} $. Hence, we have
		\begin{flalign*}
		& (M'_{T|P}(Z^{-1}A_{I|R}))^{-1}D_{T|P}(Z^{-1}A_{I|R})=\\
		& (Z^{-1}M'_{T|P}A_{I|R})^{-1}Z^{-1}D_{T|P}A_{I|R}=\\
		& (M'_{T|P}A_{I|R})^{-1}ZZ^{-1}D_{T|P}A_{I|R}=\\
		& D_{T|P}((M'_{T|P}A_{I|R})^{-1}A_{I|R}).
		\end{flalign*}
		Finally, this calculation shows that the composition $g^*_{I|R,J|S}og^*_{J|S,T|P}$ is obtained by
		\begin{equation*}
		D_{T|P}((M'_{T|P}A_{I|R})^{-1}A_{I|R})=D_{T|P}A_{T|P},
		\end{equation*}
		and it completes the proof. 
	\end{enumerate}
\end{proof}
In the next lemma, we specify $ \tilde{g}_{I|R,J|S} $,  the map which is introduced in \eqref{equg}.
\begin{lem}\label{lem2.2}
	Let $g^*_{I|R,J|S}$ be the gluing morphism between $\mathcal{O}_{J|S}|_{U_{{J|S},{I|R}}}$ and $\mathcal{O}_{I|R}|_{U_{{I|R},{J|S}}}$. Then, there is a morphism $(\tilde{g}_{I|R,J|S},\tilde{g}^*_{I|R,J|S})$ between $(U_{{I|R},{J|S}}, C^{\infty}_{I|R}|_{U_{{I|R},{J|S}}})$ and $(U_{{J|S},{I|R}}, C^{\infty}_{J|S}|_{U_{J|S,I|R}})$ induced by  $g^*_{I|R,J|S}$ such that $\tilde{g}_{I|R,J|S}$ is a map from $U_{{I|R},{J|S}}$ to $U_{{J|S},{I|R}}$.
\end{lem}
\begin{proof}
	Since $ g_{I|R,J|S}^* $ is a morphism between sheaves of rings, we have
	\begin{equation*}
	g_{I|R,J|S}^* (\mathcal{J}_{J|S}) \subset \mathcal{J}_{I|R},
	\end{equation*}
	where $ \mathcal{J} $ is the sheaf of nilpotent elements of $ \mathcal{O} $. So $ g_{I|R,J|S}^* $ induces the following map
	\begin{equation*}
	\bar{g}_{I|R,J|S}^*: \frac{\mathcal{O}_{J|S}}{\mathcal{J}_{J|S}}\bigg|_{U_{{J|S},{I|R}}} \longrightarrow \frac{\mathcal{O}_{I|R}}{\mathcal{J}_{I|R}}\bigg|_{U_{{I|R},{J|S}}}.
	\end{equation*}
	On the other hand, one has the following isomorphism:
	\begin{align*}
	\tau_{I|R}: & \frac{\mathcal{O}_{I|R}}{\mathcal{J}_{I|R}} \rightarrow C^{\infty}_{I|R},\\
	&  s+\mathcal{J}_{I|R} \longmapsto \tilde{s},
	\end{align*}
	where $ \tilde{s}(x) $, for $ x \in U_{I|R} $, is a unique real number such that $ s- \tilde{s}(x) $ is not invertible in $ \mathcal{O}(U) $ for all open neighborhoods $U$ of $x$. So one may consider the composition $ \tilde{g}^*_{{I|R},{J|S}}= \tau_{I|R}\circ
	\bar{g}^*_{{I|R},{J|S}}\circ \tau_{J|S}^{-1} $ as a map between sheaves of rings of smooth functions  $ C^{\infty}_{J|S} $ and $ C^{\infty}_{I|R} $. Thus, there is a smooth map $ \tilde{g}_{I|R,J|S}: U_{I|R} \rightarrow U_{J|S} $ such that $ \tilde{g}_{I|R,J|S}^*(f)= f \circ \tilde{g}_{I|R,J|S} $  for each $ f\in C^{\infty}_{J|S}  $ (c.f \cite{f31}).
\end{proof}
\begin{prop}
	If $ (\mathcal{X},\mathcal{O}) $ is the ringed space which is constructed by gluing $ (U_{I|R},\mathcal{O}_{I|R}) $ through $ g_{I|R,J|S} $ then the reduced manifold associated to it, i.e. $ (\mathcal{X},\tilde{\mathcal{O}}) $, is a manifold diffeomorphic to $ G_k(\mathbb{R}^m)\times G_l(\mathbb{R}^n) $.
\end{prop}
Before proving the proposition, we need the following lemma.
\begin{lem}\label{lem1} There exists an injective immersion from $ \mathcal{X} $ to $ G_{k+l}(\mathbb{R}^{m+n}) $. \end{lem}
\begin{proof}
	First, note that $ \mathcal{X} $ is constructed by gluing $\nu$-domains $U_{I|R}$ through $\tilde{g}_{{I|R},{J|S}}$. 
The elements of $ U_{I|R} $ may be denoted by a $(k+l)\times (m+n-k-l)$ matrix as follows:
	\begin{equation*}
	\left[
	\begin{array}{cc}
	X_I & 0 \\
	0 & X_R
	\end{array}
	\right].
	\end{equation*}
	Now suppose that $ V_{\tilde{J}} $ is an open neighborhood in $ G_{k+l}(\mathbb{R}^{m+n}) $ where $ \tilde{J}\subset \{1, \cdots , m+n \} $ is a multi index with k+l elements. We are going to define a map $ \psi_{I|R,\tilde{J}}:U_{I|R}\longrightarrow V_{\tilde{J}} $.\\
	For this purpose, first decompose $ \tilde{J} $ into two parts, namely $J, S$ as follows:
	
	Suppose that $ j_r $ is the largest element in $ \tilde{J} $ smaller than or equal to $ m $. 
	Set
	\begin{equation*}
	J:=\{j\in \tilde{J}, j<j_r\}, \qquad
	S:= \tilde{J}\backslash J.
	\end{equation*}
	To define $ \psi_{I|R,\tilde{J}} $, suppose that $ X\in U_{I|R} $ and $ A_{I|R} $ be the label associated with $ U_{I|R} $. 
	
	Now $ \psi_{I|R,\tilde{J}}:U_{I|R}\longrightarrow V_{\tilde{J}}  $ can be defined as a map with component functions to be the same as entries of the following matrix:
	\begin{equation*}
	D_{J|S}\big( (M_{J|S}\nu''A_{I|R})^{-1}\nu''A_{I|R}\big),
	\end{equation*}
	where $ \nu'' $ is a map which is defined on homogenous element, say $ a $, by $ \nu''a= \widetilde{\nu^{p(a)}a} $. By $ \nu'' A_{I|R} $, we mean the matrix $ (\nu''a_{ij}) $ where $ A_{I|R}= (a_{ij}) $. \\
	It can be shown that the following diagram is commutative
	\begin{displaymath}
	\xymatrix{
		U_{I|R} \ar[d]_{\tilde{g}_{I|R,I'|R'}} \ar[r]^{\psi_{I|R,\tilde{J}}} & V_{\tilde{J}} \ar[d]^{\theta_{\tilde{J},\tilde{J'}}} \\
		U_{I'|R'}\ar[r]^{\psi_{I'|R',\tilde{J'}}} & V_{\tilde{J'}} }
	\end{displaymath}
	where $ \theta_{\tilde{J},\tilde{J'}} $ is the transition map which defines the rule for changing coordinates on $ G_{k+l}(\mathbb{R}^{m+n}) $ and $ \tilde{g}_{I|R,I'|R'} $ is introduced in lemma \eqref{lem2.2}. To this end, first, we compute component functions of the map $ \theta_{\bar{J}, \bar{J}'}\circ \psi_{{I|R},\bar{J}} $. These are the entries  of the following matrix:
	\begin{align*}
	& D_{J'|S'}\Bigg(\bigg(M_{J'|S'}\big((M_{J|S}\nu'' A_{I|R})^{-1}\nu'' A_{I|R}\big)\bigg)^{-1}(M_{J|S}\nu'' A_{I|R})^{-1}\nu'' A_{I|R}\Bigg)\\
	&  = D_{J'|S'}\bigg(\big((M_{J|S}\nu'' A_{I|R})^{-1}(M_{J'|S'}\nu'' A_{I|R})\big)^{-1}(M_{J|S}\nu'' A_{I|R})^{-1}\nu'' A_{I|R}\bigg)\\
	&= D_{J'|S'}\bigg((M_{J'|S'}\nu'' A_{I|R})^{-1}(M_{J|S}\nu'' A_{I|R})(M_{J|S}\nu'' A_{I|R})^{-1}\nu'' A_{I|R}\bigg)\\
	&= D_{J'|S'}\bigg((M_{J'|S'}\nu'' A_{I|R})^{-1}\nu'' A_{I|R}\bigg).
	\end{align*}
	Now, we are going to compute the following composition:
	\begin{equation*}
	\psi_{{I'|R'},\tilde{J}'}\circ \tilde{g}_{{I|R}{I'|R'}}.
	\end{equation*}
	
	First, note that the component functions of $ \tilde{g}_{{I|R},{I'|R'}} $ are among the entries of the matrix  $ \widetilde{(M_{I'|R'} A_{I|R})}^{-1} \nu'' A_{I|R} $. Thus the definition of the $ \psi_{{I'|R'},\tilde{J}'} $ implies that the component functions of the composition are the entries of the following matrix:
	\begin{align*}
	& D_{J'|S'}\Bigg( \bigg(M_{J'|S'}\big( \widetilde{(M_{I'|R'}\nu'' A_{I|R})}^{-1}\nu''A_{I|R}\big)\bigg)^{-1}\widetilde{(M_{I'|R'}\nu'' A_{I|R})}^{-1}\nu''A_{I|R}\Bigg)\\
	= &  D_{J'|S'}\bigg((M_{J'|S'}\nu'' A_{I|R})^{-1}\nu'' A_{I|R}\bigg).
	\end{align*}
	This shows that the above diagram commutes. Commutativity shows that $  \psi_{I|R,\tilde{J}} $ is a coordinate representation of a map, say $ \psi $, from $ \mathcal{X} $ to $ G_{k+l}(\mathbb{R}^{m+n}) $. 

Now, we show that the map $ \psi:\mathcal{X}\to G_{k+l}(\mathbb{R}^{m+n}) $ is 1-1 immersion.

		This is obvious that each $\psi_{I|R,\tilde{J}}$ is smooth, so $ \psi $ is a smooth map. The left inverse of $ \psi $ is a map which its component functions are among the entries of the following matrix:
		\begin{equation*}
		D_{I|R} \bigg((M_{I|R}\{Y\}_{J|S})^{-1}\{Y\}_{J|S}\bigg),
		\end{equation*}
		where $ \{Y\}_{J|S} $ is a matrix constructed by adding $ k+l $ columns with indices in $ J\cup S $ to $ Y $ which together form an identity matrix. This map is smooth and $ \chi_{\tilde{J},I|R} \circ \psi_{I|R,\tilde{J}} =Id_{U_{I|R}} $. So $ \psi $ is a 1-1 immersion.
	\end{proof}
\renewcommand\proofname{\textbf{Proof of proposition 2.3}}
\begin{proof} Consider the imbedding $$\Lambda :G_k(\mathbb{R}^m)\times G_l(\mathbb{R}^n) \rightarrow G_{k+l}(\mathbb{R}^{m+n}),$$ which is defined locally by $  \Lambda (P,Q)= \pi_1(P)+ \pi_2(Q) $
	where $ \pi_1:G_k(m)\to G_{k+l}(m+n)$  and $\pi_2:G_l(n)\to G_{k+l}(m+n) $ are the maps induced by maps $\mathbb{R}^m\to \mathbb{R}^{m+n}$ with $ (a_1, \cdots, a_m)\mapsto (a_1, \cdots, a_m, 0, \cdots,0) $ and $ \mathbb{R}^n\to \mathbb{R}^{m+n}$ with $(a_1, \cdots, a_n)\mapsto (0, \cdots, 0,a_1, \cdots, a_n) $ respectively.
	
	It is easy to see that $ \Lambda $ is an embedding and  $ \Lambda (G_k(\mathbb{R}^m)\times G_l(\mathbb{R}^n)) $ is equal to the image of $ \mathcal{X} $ under $ \psi $. So there exists a unique diffeomorphism $ \bar{\Lambda}: \mathcal{X} \rightarrow G_k(\mathbb{R}^m)\times G_l(\mathbb{R}^n) $ such that $ \Lambda \circ \bar{\Lambda}= \psi $.
\end{proof}
\begin{defn}
	By a super vector bundle of rank $ k|l $ on a supermanifold $ (X,\mathcal{O}) $, we mean a locally free sheaf of $ \mathcal{O} $-modules of rank $ k|l $.
\end{defn}
\begin{thm}\label{thm}
	There exists a canonical $ k|l $-super vector bundle over $ _\nu G_{k|l}(m|n) $.
\end{thm}
\renewcommand\proofname{Proof}
\begin{proof}
	On every neighborhood $ U_{I|R}$, one may define a sheaf of $ \mathcal{O}_{I|R} $- modules of rank $ k|l $ as 
	\begin{equation*}
	\mathcal{O}_{I|R}\otimes_{\mathbb{R}}\left\langle A^1_{I|R}, \cdots, A^{k+l}_{I|R} \right\rangle_{\mathbb{R}},
	\end{equation*}
	where by $ A^i_{I|R} $ we mean $ i $-th row of $ A_{I|R} $ and also $ \left\langle A^1_{I|R}, \cdots, A^{k+l}_{I|R} \right\rangle_{\mathbb{R}} $ is a supervector space on $ \mathbb{R} $ generated by $  A^1_{I|R}, \cdots, A^{k+l}_{I|R} $. These sheaves may be glued together through the morphisms
	\begin{equation*} 
	\eta_{I|R,J|S}:\mathcal{O}_{J|S}|_{U_{{J|S},{I|R}}}\otimes_{\mathbb{R}}\left\langle A^1_{J|S}, \cdots, A^{k+l}_{J|S} \right\rangle_{\mathbb{R}} \longrightarrow \mathcal{O}_{I|R}|_{U_{{I|R},{J|S}}}\otimes_{\mathbb{R}}\left\langle A^1_{I|R}, \cdots, A^{k+l}_{I|R} \right\rangle_{\mathbb{R}},
	\end{equation*}
	which are defined by
	\begin{equation*}
	a\otimes A^i_{J|S}\longmapsto g^*_{{I|R,J|S}}(a)\sum\limits_t m_{it}|_{U_{I|R,J|S}}\otimes A^t_{I|R},
	\end{equation*}
	where $ m_{it} $ is the entry of $ (M'_{J|S}A_{I|R})^{-1} $ on the $ i $-th row and $ t $-th column.
	A straight forward computation shows that $ \eta_{{I|R},{J|S}} $ satisfies the conditions of proposition \eqref{prop1}.
\end{proof}
This canonical super vector bundle over $_{\nu}G_{k|l}(m|n)$ is denoted by $\Gamma.$ It is worth to know that there exists a $2$-cocycle associated to $\Gamma$ canonically. Indeed, one may set $$h_{{I|R}{J|S}}:=(\nu_\circ)^{p(I|R)+p(J|S)}(M^\prime_{I|R}A_{J|S})^{-1},$$
where $p(I|R) = 0$ if $I|R$ is a standard index and otherwise
$p(I|R) = 1$. Moreover, $A_{I|R}$ is the supermatrix associated to the $\nu$-domain $U_{I|R}$. It can be shown that ${h_{{I|R}{J|S}}}$ is a cocycle of open cover $U_{{I|R}{J|S}}$. This shows that $\Gamma$ and eventually $_{\nu}G_{k|l}(m|n)$ are novel things different from the supergrssmannian and its canonical super vector bundle. See \cite{Roshandel} for more details.

\end{document}